\documentclass[12pt]{article}
\usepackage{amssymb,amsmath, amsthm}
\usepackage{color,xcolor}
\newcommand{\eb}{\begin{equation}}
\newcommand{\ee}{\end{equation}}
\newcommand{\ebx}{\begin{equation*}}
\newcommand{\eex}{\end{equation*}}
\newcommand{\dotproduct}{\langle\cdot,\cdot\rangle}
\newcommand{\dotproductz}{\langle z,\bar z\rangle}
\newtheorem{lemma}{Lemma}[section]
\newtheorem{proposition}[lemma]{Proposition}
\newtheorem{theorem}[lemma]{Theorem}

\newtheorem{corollary}[lemma]{Corollary}
\newtheorem{definition}[lemma]{Definition}

\newtheorem{conjecture}[lemma]{Conjecture}

\allowdisplaybreaks

\makeatletter
\renewcommand*\env@matrix[1][*\c@MaxMatrixCols c]{%
  \hskip -\arraycolsep
  \let\@ifnextchar\new@ifnextchar
  \array{#1}}
\makeatother

\begin{document}

\title{On the rank of Hermitian polynomials\\ and the SOS Conjecture}

\author{Yun Gao\footnote{School of Mathematical Sciences, Shanghai Jiao Tong University, Shanghai, People's Republic of China. \textbf{Email:}~gaoyunmath@sjtu.edu.cn}, Sui-Chung Ng\footnote{School of Mathematical Sciences, Shanghai Key Laboratory of PMMP, East China Normal University, Shanghai, People's Republic of China. \textbf{Email:}~scng@math.ecnu.edu.cn}}

\maketitle

\begin{abstract}
Let $z\in\mathbb C^n$ and $\|z\|$ be its Euclidean norm. Ebenfelt proposed a conjecture regarding the possible ranks of the Hermitian polynomials in $z,\bar z$ of the form $A(z,\bar z)\|z\|^2$, known as the \textit{SOS Conjecture}, where SOS stands for ``sums of squares". In this article, we employed and extended the recent techniques developed for local orthogonal maps to study this conjecture and its generalizations to arbitrary signatures. 

We proved that for a given Hermitian form $\langle\cdot,\cdot\rangle_\star$ on $\mathbb C^n$ of any signature with at least two non-zero eigenvalues, the ranks of the Hermitian polynomials of the form $A(z,\bar z)\|z\|^2_\star$ do not lie in a finite number of non-trivial gaps on the real line. Not only is this consistent with the SOS Conjecture, but it also demonstrates that in fact the conjecture qualitatively holds in all signatures except the trivial ones. In addition, these gaps can be explicitly calculated in terms of $n$ and the signature of $\langle\cdot,\cdot\rangle_\star$, and the number of such gaps and their sizes are of the same order as those stated in the SOS Conjecture.
\end{abstract}

\section{Introduction}

Denote by $\|z\|$ the Euclidean norm of a point $z\in\mathbb C^n$. For a Hermitian polynomial $A(z,\bar z)\in\mathbb C[z_1,\ldots,z_n,\bar z_1,\ldots,\bar z_n]$, i.e. when $A(z,\bar z)$ is real valued, the problem about whether or not $A(z,\bar z)\|z\|^2$ can be written as a sum of norm squares of polynomials in $z$ has attracted considerable interest. Sometimes it is called the \textit{sums of squares (SOS) problem}. The problem is partly motivated by the Hermitian analogues of Hilbert's 17th problem and the study of proper holomorphic maps between the complex unit balls.  Readers may see references~\cite{Da}, \cite{CD1}, \cite{CD2} for these connections.

More recently, a related problem regarding the \textit{rank} of $A(z,\bar z)\|z\|^2$ has been posed by Ebenfelt~\cite{Eb}. Suppose $A(z,\bar z)\|z\|^2$ can be written as a sum of norm squares of polynomials, the rank (or linear rank) of $A(z,\bar z)\|z\|^2$ is defined to be the dimension of the vector space spanned by these polynomials, or equivalently, the minimum number of such polynomials. Motivated by the \textit{Gap Conjecture} on the rational proper maps between the complex unit balls, Ebenfelt proposed a conjecture regarding the possible ranks of $A(z,\bar z)\|z\|^2$ when $A$ varies. (For the detail of the Gap Conjecture, we refer to reader to the original article of Huang-Ji-Yin~\cite{HJY}.) Using a CR version of the Gauss equation, Ebenfelt proved that the Gap Conjecture is a consequence of the following conjecture on the rank of $A(z,\bar z)\|z\|^2$, which he named the \textit{SOS Conjecture}:

\begin{conjecture}[\cite{Eb}]\label{sos conjecture}
Let $n\geq 2$ and $k_0$ be the largest integer $k$ such that $n>k(k+1)/2$. Let $R$ be the rank of $A(z,\bar z)\|z\|^2$, where $z\in\mathbb C^n$, then either
$$
R\geq (k_0+1)n-\dfrac{k_0(k_0+1)}{2}-1,
$$
or there exists $k\in\{1,\ldots, k_0\}$ such that
$$
kn-\dfrac{k(k-1)}{2}\leq R\leq kn.
$$
\end{conjecture}

There have been a number of articles providing various sorts of evidence for the SOS conjecture, including the case of monomials in $\mathbb C^3$ by Brooks-Grundmeier~\cite{BG}, the cases where $A(z,\bar z)$ itself is a sum of squares by Grundmeier-Halfpap~\cite{GH} or is of certain special signatures by Brooks-Grundmeier-Schenck~\cite{BGS}. 

In this article, we will study the SOS conjecture and its generalizations. We start by remarking that the diagonalizability of Hermitian matrices implies that even in the cases where $A(z,\bar z)\|z\|^2$ cannot be written as a sum of norm squares of polynomials, it can always be written as a difference of two sums of squares
\begin{equation}\label{intro eq}
A(z,\bar z)\|z\|^2=|P^+_1|^2+\cdots+|P^+_p|^2-|P^-_1|^2-\cdots-|P^-_q|^2,
\end{equation}
where $P^+_1,\ldots,P^+_p,P^-_1,\ldots,P^-_q$ are polynomials in $z$. In addition, the dimension of the vector space spanned by these polynomials is uniquely determined by the function $A(z,\bar z)\|z\|^2$ and will be defined to be the rank of $A(z,\bar z)\|z\|^2$. In fact, the notion of rank is similarly defined for any Hermitian polynomial. Therefore, one may also study a more general rank problem in which $\|z\|^2=\dotproductz$ is replaced by other Hermitian forms of any signatures. More concretely, we want to study the possible ranks of the functions of the form $A(z,\bar z)\|z\|_\star^2$, where $\|z\|_\star^2:=\dotproductz_\star$ and $\dotproduct_\star$ is a Hermitian form on $\mathbb C^n$ which can be indefinite or degenerate. One may call it the \textit{Generalized SOS problem}.

The starting point of our approach to the Generalized SOS problem is to homogenize and polarize Eq.(\ref{intro eq}), which then gives us an \textit{orthogonal rational map} $F$ (Definition~\ref{orthogonal rational map}) between two projective spaces. Roughly speaking, it means $F$ is a rational map preserving certain orthogonality defined on the projective spaces using Hermitian forms. Therefore, we are able to apply the techniques developed in~\cite{GN}.
The basic philosophy here is the same. That is, we use a hyperplane restriction theorem (Theorem~\ref{dim thm}) to get estimates for the dimensions of the linear spans of the images under $F$ of linear subspaces in every dimension. These estimates will then be combined with the dimension restrictions given by orthogonality to show that the dimension of the linear span of the image of $F$, which is just the rank of $A(z,\bar z)\|z\|_\star^2$, cannot lie in certain intervals. However, in contrast to the non-degenerate situation handled in~\cite{GN}, we will allow degeneracy in the Hermitian forms here. This is not just a matter of generalization but is actually something one cannot avoid. The reason is that even if we start with non-degenerate Hermitian forms, degeneracy will always show up in the end as the new variable introduced during the homogenization of Eq.(\ref{intro eq}) will not appear in the Hermitian form of the source projective space. The degeneracy of the Hermitian form will cause a number of new difficulties, mainly due to the fact that a general hyperplane is not an orthogonal complement anymore.

Our main result is as follows: 

\begin{theorem}\label{general thm}
Let $\dotproduct_\star$ be a Hermitian form on $\mathbb C^n$ of an arbitrary signature and $\tau$ be the multiplicity of its zero eigenvalue.
Let $k_0$ be the largest $k\in\mathbb N$ such that $n\geq k^2+k(3+2\tau)+2+\tau$. For a Hermitian polynomial $A(z,\bar z)\in\mathbb C[z_1,\ldots,z_n,\bar z_1,\ldots,\bar z_n]$, where $z=(z_1,\ldots,z_{n})$, let $R$ be the rank of $A(z,\bar z)\|z\|_\star^2$. Then, either
$$
R\geq (k_0+1)n-(k_0+1)(k_0+\tau)
$$
or there exists $k\in\{1,\ldots,k_0\}$ such that
$$
kn-k(k-1+\tau)\leq R\leq kn+k(2+\tau)
$$
\end{theorem}

We remark that Theorem~\ref{general thm} demonstrates that in fact Conjecture~\ref{sos conjecture} qualitatively holds in \textit{all} signatures except the trivial ones, i.e. except those with $\tau\geq n-1$.

When $\tau=0$, the Hermitian form $\langle\cdot,\cdot\rangle_\star$ is non-degenerate, it is a situation which in particular includes the Euclidean case. Theorem~\ref{general thm} thereby gives the following corollary which contains a partial confirmation of Conjecture~\ref{sos conjecture}.

\begin{corollary}\label{sos cor}
Follow the notations in Theorem~\ref{general thm}, suppose $n\geq 2$ and $\dotproduct_\star$ is non-degenerate. Let $\kappa_0$ be the largest $\kappa\in\mathbb N$ such that $n\geq (\kappa+2)(\kappa+1)$. Then, either
$$
R\geq (\kappa_0+1)n-(\kappa_0+1)\kappa_0
$$
or there exists $\kappa\in\{1,\ldots,\kappa_0\}$ such that
$$
\kappa n-\kappa(\kappa-1)\leq R\leq \kappa n+2\kappa.
$$
\end{corollary}

\noindent\textbf{Remark.} If $A(z,\bar z)$ is bihomogeneous, then it will follow directly from Theorem~\ref{homo thm} that the last inequality in Corollary~\ref{sos cor} can be improved to
$$
\kappa n-\kappa(\kappa-1)\leq R\leq \kappa n.
$$

%-----------------------------------------------------------------------------------------

\section{Orthogonality and Generalized SOS problem}

We first bring up the following hyperplane restriction theorem for the local holomorphic maps between projective spaces established in~\cite{GN}, which is very useful for the mapping problems between real hyperquadrics. It is inspired by an analogous theorem of Green for $H^0(\mathcal O_{\mathbb P^n}(d))$, $d\geq 1$. To state the theorem, let us recall the Macaulay representations of a positive integer. Fix $A\in\mathbb N^+$. For every $n\in\mathbb N^+$, there exist unique positive integers $a_n>a_{n-1}>\cdots >a_\delta$, where $\delta\geq 1$ and $a_j\geq j$ for every $j$, such that $A=\binom{a_n}{n}+\cdots+\binom{a_\delta}{\delta}$. This is called the \textit{$n$-th Macaulay representation of $A$}. Using this, we define $A^{-<n>}:=\binom{a_n-1}{n-1}+\cdots+\binom{a_\delta-1}{\delta-1}$, in which we adopt the convention that $\binom{a}{b}=0$ if $a<b$ or $b=0$. 

\begin{theorem}[\cite{GN}]\label{dim thm}
Let $f:U\subset\mathbb P^n\rightarrow\mathbb P^M$ be a local holomorphic map such that 
$\dim(\mathrm{span}(f(U))\geq N$ for some positive integer $N\leq M$, where ``span'' means the projective linear span. Then, for a general hyperplane $\Pi$ such that $\Pi\cap U\neq\varnothing$, $\dim(\mathrm{span}(f(\Pi\cap U))\geq N^{-<n>}$.
\end{theorem}

To match better the notations employed in~\cite{GN}, in what follows we will work on $\mathbb C^{n+1}$ instead of $\mathbb C^n$.

Let $r,s,t\in\{0,\ldots,n+1\}$ such that $(r,s)\neq (0,0)$ and $r+s+t=n+1$. Consider the possibly indefinite or degenerate Hermitian form on $\mathbb C^{n+1}$ defined by
$$
\langle z,w\rangle_{r,s,t}:=z_1\bar w_1+\cdots+z_r\bar w_r
-z_{r+1}\bar w_{r+1}-\cdots-z_{r+s}\bar w_{r+s},
$$
where $z,w\in\mathbb C^{n+1}$. Thus, $\dotproduct_{r,s,t}$ is of signature $(r,s,t)$. Denote $\dotproductz_{r,s,t}$ by $\|z\|^2_{r,s,t}$. If $t=0$, we will simply write $\dotproduct_{r,s}$ and $\|\cdot\|_{r,s}^2$.

Let $B(z,\bar z)\in\mathbb C[z_1,\ldots,z_{n+1},\bar z_1,\ldots,\bar z_{n+1}]$ be a non-zero Hermitian \textit{bihomogeneous} polynomial, where $z=(z_1,\ldots,z_{n+1})$. Using the diagonalizability of Hermitian matrices, we can write
\begin{equation}\label{decomposition}
	B(z,\bar z)\|z\|^2_{r,s,t}=|h^+_1|^2+\cdots+|h^+_p|^2-|h^-_1|^2-\cdots-|h^-_q|^2
\end{equation}
for some linearly independent homogeneous polynomials $h^+_1,\ldots,h^+_p,h^-_1,\ldots,h^-_q$ of the same degree. These polynomials are uniquely determined by $B(z,\bar z)\|z\|^2_{r,s,t}$ up to the action of the generalized unitary group $U(p,q)$ (see~\cite{GN2} for a proof). In particular, the sum $R:=p+q$ is well defined and will be called the \textit{rank} of $B(z,\bar z)\|z\|^2_{r,s,t}$.

Now fix $r,s,t$. To study the possible ranks of $B(z,\bar z)\|z\|^2_{r,s,t}$ for a varying $B$, we use Eq.(\ref{decomposition}) to define a polynomial map $\tilde F:\mathbb C^{n+1}\rightarrow\mathbb C^R$:
$$
\tilde F(z_1,\ldots,z_{n+1}):=(h^+_1,\ldots,h^+_p,h^-_1,\ldots,h^-_q).
$$
If we polarize Eq.(\ref{decomposition}), we get
\begin{equation}\label{polarized eq}
B(z,\bar w)\langle z,w\rangle_{r,s,t}=\langle \tilde F(z),\tilde F(w)\rangle_{p,q},
\end{equation}
from which we deduce that $\langle \tilde F(z),\tilde F(w)\rangle_{p,q}=0$ for every $z,w$ such that $\langle z,w\rangle_{r,s,t}=0$. Since $\tilde F$ is given by homogeneous polynomials, it can also be regarded as a rational map $F:\mathbb P^n\dashrightarrow\mathbb P^{R-1}$. Note that although the Hermitian forms themselves are no longer defined on the projective spaces, the associated orthogonality still makes sense. Such a map between two projective spaces preserving orthogonality has been studied in~\cite{GN1} and is called an \textit{orthogonal map}. For the purpose of this article, it suffices for us to use the following definition specialized for the present context:

\begin{definition}\label{orthogonal rational map}
Let $n,N\in\mathbb N^+$ and $r,s,t,p,q\in\mathbb N$ such that $r+s+t=n+1$ and $p+q=N+1$. For $z,w\in\mathbb P^n$, where $z=[z_1,\ldots,z_{n+1}]$ and $w=[w_1,\ldots,w_{n+1}]$, we write $z\perp_{r,s,t} w$ if $\langle z,w\rangle_{r,s,t}=0$ and similarly for $z'\perp_{p,q}w'$ when $z',w'\in\mathbb P^N$.
A rational map $F:\mathbb P^n\dashrightarrow\mathbb P^N$ is called an \textbf{orthogonal rational map} with respect to $(\perp_{r,s,t},\perp_{p,q})$ if $F(z)\perp_{p,q}F(w)$ for every $z,w$ in the domain of definition of $F$ such that $z\perp_{r,s,t}w$.
\end{definition}

We are now able to give a lower bound for the rank of $B(z,\bar z)\|z\|_{r,s,t}^2$ using orthogonality and Theorem~\ref{dim thm}:

\begin{proposition}\label{lower bound}
Let $z\in\mathbb C^{n+1}$ and $R$ be the rank of $B(z,\bar z)\|z\|^2_{r,s,t}$, where $B$ is non-zero and Hermitian bihomogeneous. Then $R\geq r+s$.
\end{proposition}
\begin{proof}
We first claim that it suffices to prove the proposition for the case $t=0$. To see this, instead of a rational map between projective spaces, we for the moment regard the map $F$ induced by $B(z,\bar z)\|z\|_{r,s,t}^2$ as a map from $\mathbb C^{n+1}$ to $\mathbb C^R$. Note that the restriction of $\langle\cdot,\cdot\rangle_{r,s,t}$ on a general $(r+s)$-dimensional linear subspace of $\mathbb C^{n+1}$ is of signature $(r,s,0)$ and thus by restricting everything to a general $(r+s)$-dimensional linear subspace and performing a linear change of coordinates, we get a non-zero polynomial map $F^\flat$ from $\mathbb C^{r+s}$ to $\mathbb C^R$ which satisfies the polarized equation
$
B^\flat(z,\bar w)\langle z,w\rangle_{r,s}=\langle F^\flat(z), F^\flat(w)\rangle_{p,q}
$
(as Eq.(\ref{polarized eq}) above) for some Hermitian bihomogeneous polynomial $B^\flat(z,\bar w)\in\mathbb C[z_1,\ldots,z_{r+s},\bar w_1,\ldots,\bar w_{r+s}]$. Since $F^\flat$ is obtained from $F$ by restriction and a change of domain coordinates, it follows that the rank of $B^\flat(z,\bar z)\|z\|_{r,s}^2$ is not greater than that of $B(z,\bar z)\|z\|^2_{r,s,t}$ and thus it suffices to show that the rank of $B^\flat(z,\bar z)\|z\|_{r,s}^2$ is at least $r+s$.

Thus, from now on we can assume that $t=0$. As before, since $B\neq 0$, we get from $B(z,\bar z)\|z\|_{r,s}^2$ an orthogonal rational map $F:\mathbb P^n\dashrightarrow\mathbb P^{R-1}$ with respect to $(\perp_{r,s},\perp_{p,q})$ for some $p+q=R$, and the image of $F$ is not contained in any hyperplane. Now Theorem~\ref{dim thm} implies that 
\begin{equation}\label{first estimate}
\dim(\mathrm{span}(F(\Pi))\geq (R-1)^{-<n>}
\end{equation}
for a general hyperplane $\Pi\subset\mathbb P^n$. On the other hand, since $\dotproduct_{r,s}$ is non-degenerate, it follows that every hyperplane in $\mathbb P^n$ is orthogonal to a point in $\mathbb P^n$. Consequently, from the polarized equation~(\ref{polarized eq}), we see that $F$ maps hyperplanes to hyperplanes. Thus, combining with (\ref{first estimate}) we get
\begin{equation}\label{second estimate}
(R-1)^{-<n>}\leq R-2.
\end{equation}
If $R<r+s$, then 
$$
R-1\leq r+s-2=n-1
$$
and hence the $n$-th Macaulay representation of $R-1$ is $\binom{n}{n}+\cdots\binom{\delta}{\delta}$ for some $\delta\geq 2$. Now it follows from the definition of the operator $\cdot^{-<n>}$ that 
$$
(R-1)^{-<n>}=R-1,
$$
which contradicts (\ref{second estimate}). Therefore, we must have $R\geq r+s$.
\end{proof}

\begin{proposition}\label{contradict}
Let $n,N\in\mathbb N^+$ and $r,s,t,p,q\in\mathbb N$ such that $r+s+t=n+1$ and $p+q=N+1$. Let $F:\mathbb P^n\dashrightarrow\mathbb P^N$ be an orthogonal rational map with respect to $(\perp_{r,s,t},\perp_{p,q})$ and $D_m:=\dim(\mathrm{span}(F(M)))$ for a general $m$-dimensional linear subspace $M\subset\mathbb P^n$. Then, for any $m_1,m_2\in\mathbb N$ such that $m_1+m_2\leq r+s-2$, we have
$$
D_{m_1}+D_{m_2}\leq N-1.
$$
\end{proposition}
\begin{proof}
Let $G_m$ be the Grassmannian of $m$-dimensional linear subspaces in $\mathbb P^n$ and $Z_m\subset G_m$ be a proper complex analytic subvariety such that $\dim(\mathrm{span}(F(M)))=D_m$ for every $M\in G_m\setminus Z_m$. For any $m_1,m_2\in\mathbb N$ such that $m_1+m_2\leq r+s-2$,  we claim that there exists $(M_1,M_2)\in (G_{m_1}\setminus Z_{m_1})\times (G_{m_2}\setminus Z_{m_2})$ such that $M_1\perp_{r,s,t} M_2$. Assume this claim for the moment. Since $F$ is orthogonal, it follows that $\mathrm{span}(F(M_1))\perp_{p,q}\mathrm{span}(F(M_2))$. As $p+q=N+1$, by lifting these two orthogonal linear subspaces in $\mathbb P^N$ to $\mathbb C^{N+1}$ (equipped with $\langle\cdot,\cdot\rangle_{p,q}$), we get the desired result 
$$
N+1\geq (\dim(\mathrm{span}(F(M_1)))+1)+(\dim(\mathrm{span}(F(M_2)))+1)=D_{m_1}+D_{m_2}+2.
$$

To prove the claim above, consider the Euclidean space $\mathbb C^{n+1}$ equipped with the Hermitian form $\langle\cdot,\cdot\rangle_{r,s,t}$. Since $m_1+m_2+2\leq r+s$, there exist orthogonal linear subspaces $\tilde M_1, \tilde M_2\subset\mathbb C^{n+1}$ such that $\dim(\tilde M_1)=m_1+1$, $\dim(\tilde M_2)=m_2+1$ and $\tilde M_1,\tilde M_2$ are non-degenerate subspaces (i.e. the restrictions of $\langle\cdot,\cdot\rangle_{r,s,t}$ on $\tilde M_1,\tilde M_2$ are non-degenerate.) Since non-degeneracy is an open condition, we can always choose $\tilde M_1$ such that the projectivization $M_1:=\mathbb P\tilde M_1$ does not lie on $Z_{m_1}$. Let $\mathcal U_1\subset (G_{m_1}\setminus Z_{m_1})$ be a neighborhood of $M_1$ such that the lifting of every $M_1'\in\mathcal U_1$ to $\mathbb C^{n+1}$ is non-degenerate. Now if $M_2:=\mathbb P\tilde M_2$ happens to lie on $Z_{m_2}$, we pick another non-degenerate $\tilde M_2'$ nearby such that $M_2':=\mathbb P\tilde M_2'\not\in Z_{m_2}$. Since every linear subspace involved is non-degenerate with respect to $\dotproduct_{r,s,t}$, by projecting to the first $r+s$ coordinates in $\mathbb C^{n+1}$, we see that as long as $\tilde M_2'$ has been chosen sufficiently close to $\tilde M_2$, we can find some $(m_1+1)$-dimensional linear subspace $\tilde M_1'\subset\mathbb C^{n+1}$ orthogonal to $\tilde M_2'$ such that $M_1':=\mathbb P\tilde M_1'\in\mathcal U_1$. Then, $(M_1', M_2')\in (G_{m_1}\setminus Z_{m_1})\times (G_{m_2}\setminus Z_{m_2})$ is a pair of orthogonal linear subspaces in $\mathbb P^n$ with respect to $\perp_{r,s,t}$.
\end{proof}

Next, we recall some notations and a couple of computations done in~\cite{GN}, which will be used in the proofs of our main theorems.

Let $n,N\in\mathbb N$ such that $n+1\leq N<\binom{n+2}{2}=\binom{n+2}{n}$. By considering the $n$-th Macaulay representation of $N$, we deduce that $N$ is of the following form:
$$
N=N(n;a,b):=\binom{n+1}{n}+\cdots+\binom{n-a+1}{n-a}+ b
$$
for some integers $a,b\geq 0$ such that $b\leq n-a-1$. In fact, the $n$-th Macaulay's representation of $N(n;a,b)$ for $b\geq 1$ is
$$
N(n;a,b)=\binom{n+1}{n}+\cdots+\binom{n-a+1}{n-a}+ \binom{n-a-1}{n-a-1}+\cdots+\binom{n-a-b}{n-a-b}.
$$
and
$$
N(n;a,0)=\binom{n+1}{n}+\cdots+\binom{n-a+1}{n-a}.
$$

\begin{lemma}[\cite{GN}]\label{nab}
$$
N(n;a,b)^{-<n>}=\left\{
\begin{matrix}[lcl]
N(n-1;a,b)&\mathrm{if}& n-a-b\geq 2;\\
N(n-1;a,b-1)&\mathrm{if}& n-a-b=1\,\,\mathrm{and}\,\,b\geq 1. %;\\
%N(n-1;a-1,0)&\mathrm{if}& n-a-b=1\,\,\mathrm{and}\,\,b=0;
\end{matrix}
\right.
$$
\end{lemma}
\begin{proof}
Follows directly from the $n$-th Macaulay representation of $N(n;a,b)$.
\end{proof}

\begin{proposition}[\cite{GN}]\label{dim prop}
Let $n\in\mathbb N^+$ and $a,b\in\mathbb N$. Let $g:U\subset\mathbb P^n\rightarrow\mathbb P^{N(n;a,b)}$ be a local holomorphic map whose image is not contained in any proper linear subspace. Let $D_m=\dim(\mathrm{span}(g(M\cap U)))$ for a general $m$-dimensional linear subspace $M$ intersecting $U$. Then,
$$
D_m\geq\left\{
\begin{matrix}[lcl]
N(m;a,b) &\,\,\mathrm{if}& a+b+1\leq m\leq n-1;\\
N(m;a,m-a-1) &\,\,\mathrm{if}& a+1\leq m\leq a+b %\,\,\,\,\,\mathrm{when}\,\,\,\,\, b\geq 1.%\\
%N(m;m-1,0) &\,\,\mathrm{if}& m\in [1,a] 
\end{matrix}
\right.
$$
\end{proposition}
\begin{proof}
See~\cite{GN}, Proposition 3.4 therein.
\end{proof}

We are now ready to prove the key statement that will lead to our results for the Generalized SOS problem.

\begin{theorem}\label{gap thm}
Let $n,N\geq 1$ and $r,s,t,p,q\in\mathbb N$ such that $r+s+t=n+1$ and $p+q=N+1$. Let $F:\mathbb P^n\dashrightarrow\mathbb P^N$ be an orthogonal rational map with respect to $(\perp_{r,s,t},\perp_{p,q})$. If there exists a non-negative integer $a$ such that
%$$(a+1)(n+t+1)\leq N\leq (a+2)n-(a+2)(a+t)-2.$$
\begin{equation}\label{hypo ineq}
(a+1)(n+t+1)\leq N\leq (a+2)(n-a-t)-2,
\end{equation}
then the image of $F$ lies in a hyperplane.
\end{theorem}
\begin{proof}
We will prove by contradiction and assume that $N$ satisfies (\ref{hypo ineq}).
Since 
$$
(a+1)(n+t+1)=\binom{n+1}{n}+\cdots+\binom{n-a+1}{n-a}+\dfrac{(a+1)(a+2t)}{2}
$$ 
and
$$
(a+2)(n-a-t)-2=\binom{n+1}{n}+\cdots+\binom{n-a+1}{n-a}
+\left(n-\dfrac{a^2+(5+2t)a+6+4t}{2}\right),
$$
we see that (\ref{hypo ineq}) is possible only if
$$
\dfrac{(a+1)(a+2t)}{2}\leq n-\dfrac{a^2+(5+2t)a+6+4t}{2}.
$$
Moreover, since the right hand side is not greater than $n-a-1$ and we have assumed that $N$ satisfies (\ref{hypo ineq}), we deduce that
$$
N=\binom{n+1}{n}+\cdots+\binom{n-a+1}{n-a}+b=N(n;a,b),
$$
for some $b$ satisfying
\begin{equation}\label{ineq 1}
\dfrac{(a+1)(a+2t)}{2}\leq b\leq n-\dfrac{a^2+(5+2t)a+6+4t}{2}.
\end{equation}

Let 
$
n_1:=\left[\dfrac{r+s}{2}\right]-1
\textrm{\,\,\,\,\,\,\,\,\,and\,\,\,\,\,\,\,\,}
n_2:=\left\{
\begin{matrix}n_1&\mathrm{ if\,\,} r+s \mathrm{\,\, is\,\, even;}\\
n_1+1&\mathrm{ if\,\,} r+s \mathrm{\,\, is\,\, odd.}
\end{matrix}\right.
$

Then, $n_1+n_2+t+1=n$ and hence
\begin{eqnarray*}
2n_1&\geq&n_1+n_2-1\\
&=&n-2-t\\
&\geq&\dfrac{(a+1)(a+2t)}{2}+\dfrac{a^2+(5+2t)a+6+4t}{2}-2-t
\,\,\,\,\,\,\,\,\,\,\,\,\,\,\,\,\,\,\,\,\textrm{(by~(\ref{ineq 1}))}\\
&=&a^2+a\left(3+2t\right)+1+2t.
\end{eqnarray*}

When $t\geq 1$, it follows immediately that $n_1\geq a+1$. On the other hand, when $t=0$, we get  $2n_1\geq a^2+3a+1=2a+(a^2+a+1)$, which as an inequality in integers, implies again that $n_1\geq a+1$. Therefore, we always have
\begin{equation}\label{n1 ineq}
n_1-a-1\geq 0.
\end{equation}
We now separate the argument into two cases:
$$
\mathbf{ Case\,\, I:\,\,}b\leq n_1-a-1\,\,\,\,\,\,\,\mathrm{and}\,\,\,\,\,\,\,
\mathbf{ Case\,\, II:\,\,}n_1-a\leq b.
$$
Let $D_m$ be the dimension of the linear span of the image under $F$ of a general $m$-dimensional linear subspace.

In \textbf{Case I}, since $n-1\geq n_2\geq n_1\geq a+b+1$, so by Proposition~\ref{dim prop} and (\ref{ineq 1}), 
\begin{eqnarray*}
D_{n_1}+D_{n_2}&\geq& N(n_1;a,b)+N(n_2;a,b)\\
&=&(a+1)(n_1+n_2+2-a)+2b\\
&=&(a+1)(r+s-a)+2b\\
&=&(a+1)(n+1-\frac{a}{2}-\frac{a+2t}{2})+2b\\
&=&(a+1)(n+1-\dfrac{a}{2})+b+\left(b-\dfrac{(a+1)(a+2t)}{2}\right)\\
&\geq&(a+1)(n+1-\dfrac{a}{2})+b\\
%&=&\binom{n+1}{n}+\cdots+\binom{n-a+1}{n-a}+b\\
&=&N(n;a,b)=N,
\end{eqnarray*}
which contradicts Proposition~\ref{contradict} since $n_1+n_2=n-1-t=r+s-2$.

%$$
%n_1-a\leq b \leq n-\dfrac{a^2+5a+6}{2}\leq 2n_1+2+t-\dfrac{a^2+5a+6}{2}
%$$ 
%and hence
%$$
%n_1-a\leq b \leq 2n_1-\dfrac{a^2+5a+2-2t}{2}
%$$
%which is equivalent to 
%$$
%\dfrac{b}{2}+\dfrac{a^2+5a+2-2t}{4}\leq n_1\leq a+b.
%$$

In \textbf{Case II},  we have $a+1\leq n_1\leq a+b$, in which the first inequality is just (\ref{n1 ineq}). Thus, by Proposition~\ref{dim prop}, 
$$
D_{n_1}\geq N(n_1;a;n_1-a-1)=(a+2)n_1-\dfrac{a^2+a}{2}
$$
and
$$
D_{n_2}\geq\left\{
\begin{matrix}
N(n_2;a;b) &\mathrm{\,\,if}& n_1=a+b &\textrm{and}& n_2=n_1+1;\\
N(n_2;a;n_2-a-1) &\mathrm{\,\,if}& n_1<a+b&\textrm{or}& n_2=n_1.
\end{matrix}
\right.
$$
Therefore,
{\small
$$
D_{n_1}+D_{n_2}\geq\left\{
\begin{matrix}
(a+2)n_1-\dfrac{a^2+a}{2}+(a+1)(n_2+1-\dfrac{a}{2})+b&\mathrm{\,\,if}& n_1=a+b &\textrm{and}& n_2=n_1+1;\\
(a+2)(n_1+n_2)-(a^2+a) &\mathrm{\,\,if}& n_1<a+b&\textrm{or}& n_2=n_1,
\end{matrix}
\right.
$$
}
which simplifies to
{\small
$$
D_{n_1}+D_{n_2}\geq\left\{
\begin{matrix}
(a+2)(n_1+n_2)-(a^2+a) &\mathrm{\,\,if}& n_1=a+b &\textrm{and}& n_2=n_1+1;\\
(a+2)(n_1+n_2)-(a^2+a) &\mathrm{\,\,if}& n_1<a+b&\textrm{or}& n_2=n_1.
\end{matrix}
\right.
$$
}
Thus, we always have 
$$
D_{n_1}+D_{n_2}\geq (a+2)(n_1+n_2)-(a^2+a)=(a+2)(n-a-t)-2\geq N,
$$
in which the last inequality follows since we have assumed (\ref{hypo ineq}). 
This again contradicts Proposition~\ref{contradict}.
\end{proof}

We can now use Theorem~\ref{gap thm} to prove our main results for the Generalized SOS problem.

\begin{theorem}\label{homo thm}
Let $\dotproduct_\star$ be a Hermitian form on $\mathbb C^n$ of an arbitrary signature and $\tau$ be the multiplicity of its zero eigenvalue.
Let $k_0$ be the largest $k\in\mathbb N$ such that $n\geq k^2+k(1+2\tau)+2+\tau$. For a Hermitian \textbf{bihomogeneous} polynomial $B(z,\bar z)\in\mathbb C[z_1,\ldots,z_n,\bar z_1,\ldots,\bar z_n]$, where $z=(z_1,\ldots,z_{n})$, let $R$ be the rank of $B(z,\bar z)\|z\|_\star^2$. Then, either
$$
R\geq (k_0+1)n-(k_0+1)(k_0+\tau)
$$
or there exists $k\in\{1,\ldots,k_0\}$ such that 
$$
kn-k(k-1+\tau)\leq R\leq kn+k\tau.
$$
\end{theorem}
\begin{proof}
By finding an orthogonal basis, we can identify $\dotproduct_\star$ with $\dotproduct_{\rho,\sigma,\tau}$ for some $\rho,\sigma,\tau$ such that $\rho+\sigma+\tau=n$.
As before (c.f.~Eq.(\ref{decomposition})), we obtain from $B(z,\bar z)\|z\|^2_{\rho,\sigma,\tau}$ an orthogonal rational map $F:\mathbb P^{n-1}\dashrightarrow\mathbb P^{R-1}$ whose image is not contained in any hyperplane, with respect to $(\perp_{\rho,\sigma,\tau},\perp_{p,q})$ for some $p,q\in\mathbb N$ such that $p+q=R$. Therefore, by Theorem~\ref{gap thm}, the integer $R-1$ cannot lie in any of the following closed intervals
$$
I_a:=[(a+1)(n+\tau), (a+2)(n-1-a-\tau)-2],\,\,\,\,\,\,\,\,\,\,\,\,\,\,a\in\mathbb N,
$$
whenever the interval is non-empty, i.e. when
$$
n\geq (a+2)(a+1+\tau)+(a+1)\tau+2.
$$
Let $a_0$ be the largest integer satisfying the last inequality. Then, by taking the complement of all non-empty $I_a$ in $\mathbb N$, it follows that there are three possibilities for $R-1$,

$(i)$ $R-1\geq (a_0+2)(n-1-a_0-\tau)-1$; 

$(ii)$ $R-1\leq n+\tau-1$; 

$(iii)$ there exists $a\in\{0,\ldots,a_0-1\}$ such that
$$
(a+2)(n-1-a-\tau)-1\leq R-1\leq (a+2)(n+\tau)-1.
$$
Let $k:=a+1$ and $k_0:=a_0+1$. Then $k_0$ is the largest integer satisfying
$$
n\geq (k+1)(k+\tau)+k\tau+2=k^2+k(1+2\tau)+2+\tau.
$$ 
Rewriting the above three possibilities in terms of $k$, we have

$(i)$ $R\geq (k_0+1)n-(k_0+1)(k_0+\tau)$; \,\,\,\,\,\,\,or

$(ii)$ $R\leq n+\tau$; \,\,\,\,\,\,\,or

$(iii)$ there exists $k\in\{1,\ldots,k_0-1\}$ such that
$$
(k+1)(n-k-\tau)\leq R\leq (k+1)(n+\tau).
$$
By Proposition~\ref{lower bound}, we necessarily have $R\geq \rho+\sigma=n-\tau$. If we combine this with $(ii)$, we see that $(iii)$ also holds for $k=0$ and hence we can now write

$(i)$ $R\geq (k_0+1)n-(k_0+1)(k_0+\tau)$; \,\,\,\,\,\,\,or

$(ii)$ there exists $k\in\{1,\ldots,k_0\}$ such that
$$
k(n-k+1-\tau)\leq R\leq k(n+\tau)
$$
and the proof is complete.
\end{proof}

\begin{proof}[Proof of Theorem~\ref{general thm}]
Similar to the proof of Theorem~\ref{homo thm}, we can identify $\dotproduct_\star$ with $\dotproduct_{\rho,\sigma,\tau}$ for some $\rho,\sigma,\tau$ such that $\rho+\sigma+\tau=n$.
If $A(z,\bar z)$ is an arbitrary Hermitian polynomial, then $A(z,\bar z)\|z\|_{\rho,\sigma,\tau}^2$ may not be bihomogeneous but we can still write
$$
A(z,\bar z)\|z\|_{\rho,\sigma,\tau}^2=|P^+_1|^2+\cdots+|P^+_p|^2-|P^-_1|^2-\cdots-|P^-_q|^2
$$
for some polynomials $P^+_1,\ldots,P^+_p,P^-_1,\ldots,P^-_q$ in $z$ as in Eq.(\ref{intro eq}) in the introduction.
We can then homogenize the equation by using the substitution $z_j=w_j/w_{n+1}$, for $1\leq j\leq n$. After multiplying the resulting equation by a factor of $|w_{n+1}|^{2d}$ for some $d\in\mathbb N$, we get an equation of Hermitian bihomogeneous polynomials of the form
$$
\tilde A(w,\bar w)\|w\|^2_{\rho,\sigma,\tau+1}=|\tilde P^+_1|^2+\cdots+|\tilde P^+_p|^2
-|\tilde P^-_1|^2-\cdots-|\tilde P^-_q|^2,
$$
where $w=(w_1,\ldots,w_{n+1})$ and $\tilde A(w,\bar w)\|w\|^2_{\rho,\sigma,\tau+1}$ is also of rank $R$. We can now apply the results for the bihomogeneous case in Theorem~\ref{homo thm}. Note that since we have introduced one more variable (namely, $w_{n+1}$), the resulting orthogonal rational map is from $\mathbb P^n$ (instead of $\mathbb P^{n-1}$ as in Theorem~\ref{homo thm}) to $\mathbb P^{R-1}$, with respect to $(\perp_{\rho,\sigma,\tau+1},\perp_{p,q})$, for some $p,q\in\mathbb N$ such that $p+q=R$. Thus, if we let $\tau'=\tau+1$ and $n'=n+1$, then $\rho+\sigma+\tau'=n'$ and hence we deduce by from Theorem~\ref{homo thm} by substitution that, if we let $k_0$ be the largest integer $k$ satisfying 
$$
n+1\geq k^2+k(1+2(\tau+1))+2+(\tau+1)=k^2+k(3+2\tau)+3+\tau,
$$
then we have either

$(i)$ $R\geq (k_0+1)(n+1)-(k_0+1)(k_0+\tau+1)=(k_0+1)n-(k_0+1)(k_0+\tau)$; \,\,\,\,\,\,\,or

$(ii)$ there exists $k\in\{1,\ldots,k_0\}$ such that
$$
\begin{matrix}[crcccl]
&k(n+1)-k(k-1+\tau+1)&\leq &R& \leq& k(n+1)+k(\tau+1)\\
\Leftrightarrow&kn-k(k-1+\tau)&\leq&R&\leq&kn+k(2+\tau).
\end{matrix}
$$
\end{proof}

\end{document}